\documentclass[11pt]{amsart}

\usepackage[margin=1in]{geometry}
\usepackage{amsmath, amsthm, amsfonts, amssymb}
\usepackage{mathrsfs}
\usepackage[all]{xy}
\usepackage{hyperref}
\hypersetup{colorlinks=true, citecolor=blue, urlcolor=black, linkcolor=blue}

\newtheorem{theorem}{Theorem}[section]
\newtheorem{proposition}[theorem]{Proposition}

\newtheorem{corollary}[theorem]{Corollary}
\newtheorem{conjecture}[theorem]{Conjecture}

\theoremstyle{definition}

\theoremstyle{remark}

\newcommand{\nc}{\newcommand}
\nc{\on}{\operatorname}
\nc{\Cl}{\on{Cl}}
\nc{\bx}{\mathbf{x}}
\nc{\bk}{\mathbf{k}}
\nc{\by}{\mathbf{y}}
\nc{\bz}{\mathbf{z}}

\DeclareMathOperator{\Gal}{Gal}

\DeclareMathOperator{\disc}{disc}

\DeclareMathOperator{\Hom}{Hom}
\DeclareMathOperator{\Sur}{Sur}
\DeclareMathOperator{\Stab}{Stab}
\DeclareMathOperator{\ind}{ind}
\DeclareMathOperator{\Ind}{Ind}
\DeclareMathOperator{\rk}{rk}

\title{Twisted Malle's conjecture}
\author{Tanav Choudhary }
\date{August 2025}

\begin{document}

\begin{abstract}
In this paper, we use inductive methods similar to those employed in a 2025 paper by Alberts, Lemke Oliver, Wang and Wood in order to prove many new cases of the Twisted Malle's Conjecture. Previously, this conjecture had only been proven for $T \trianglelefteq G$ where $T$ is abelian, and for $S_3^m \trianglelefteq S_3 \wr B$. We prove many new examples of this conjecture, such as for $A \wr M \trianglelefteq A \wr B$ where $A$ is a finite abelian $p$-group, $B$ is a $p$-group, and $M$ is a proper abelian normal subgroup of $B$, and for $C_p \wr M \trianglelefteq C_p \wr B$, where $|B|$ is odd and $M$ is a proper abelian normal subgroup of $B$. 
\end{abstract}

\maketitle

\section{Introduction}

Let $k$ be a number field, $\overline{k}$ a fixed algebraic closure of $k$, and $G$ a transitive permutation group of degree $n$. Let $G_k$ denote the absolute Galois group of $k$, and let $\Sur(G_k, G)$ denote the set of surjective continuous homomorphisms from $G_k$ to $G$. Given an element $\psi \in \Sur(G_k, G)$, $\disc_G(\psi)$ denotes the relative discriminant $\disc(F/k)$, where $F$ is the field fixed by $\pi^{-1}(\Stab_G(1))$. Let $T \trianglelefteq G$ denote a proper normal subgroup of $G$, and let $q: G \to G/T$ denote the canonical quotient map. Then, define $q_*: \Sur(G_k, G) \to \Sur(G_k, G/T)$ to be the pushforward. We now state the Twisted Malle's Conjecture, due to Alberts \cite[Conjecture 3.10]{Alberts2021}: 

\begin{conjecture}[Twisted Malle's Conjecture]
Let $k$ be a number field, and $G$ a transitive permutation group of degree $n$, and $T \trianglelefteq G$ a proper normal subgroup with canonical quotient map $q : G \to G/T$, and $\pi \in q_* \Sur(G_k, G)$. Then there exist positive constants $a,b,c > 0$ depending on $k, G, T$, and $\pi$ such that
\[
\#\{\psi \in q_*^{-1}(\pi) : |{\disc_G(\psi)}| \le X\} \sim c X^{1/a} (\log X)^{b-1}
\]
as $X \to \infty$.
\end{conjecture} 

The above conjecture can be restructured as a conjecture on counting number fields with given fixed subfields, as we will now show. Let $k$ and $G$ be as before. Given a degree $n$ field extension $L/k$, let $\tilde{L}$ denote the Galois closure of $L$ over $k$. Then, we call $L/k$ a $G$-extension if $\Gal(\tilde{L}/k)$, acting on the $n$ $k$-linear embeddings $L \hookrightarrow \tilde{L}$, is isomorphic to $G$ as a permutation group. Now, define \[ \mathcal{F}_{n, k}(G; X) = \#\{L/k: L \subset \overline{k}, [L:k] = n, \Gal(\tilde{L}/k) \cong G, |{\disc(L/k)}|\leq X\} \] where $\cong$ denotes isomorphism as permutation groups. Malle's number field counting conjecture, first formulated by Malle in \cite{Malle2002} and \cite{Malle2004}, gives a prediction for the asymptotic growth of $\mathcal{F}_{n, k}(G; X)$ as $X \to \infty$. There have been many results proving Malle's conjecture for specific transitive permutation groups $G$. See \cite[Introduction]{AlbertsLemkeOliverWangWood2025} for an overview of these results. 
\\
\\
Let $T, q$ and $q_*$ be as before. Let $L/k$ be a fixed $G/T$ extension. Define \[ N(L/k, T \trianglelefteq G; X) = \#\{F \in \mathcal{F}_{n, k}(G; X): (\tilde{F})^T = L\} .\] Concretely, $N(L/k, T \trianglelefteq G; X)$ is the number of $G$-extensions $F/k$ with $|{\disc(F/k)}| \leq X$ such that the subfield of $\tilde{F}$ fixed by $T \trianglelefteq G$ is $L$. Let $\pi \in \Sur(G_k, G/T)$ be an element corresponding to $L$ under the Galois correspondence. Then, observe that $N(L/k, T \trianglelefteq G; X)$ is equal to a constant multiple of $\#\{\psi \in q_*^{-1}(\pi) : |{\disc_G(\psi)}| \le X\}$ via the Galois correspondence. Hence, Conjecture 1.1 also gives us a prediction for the asymptotic growth of $N(L/k, T \trianglelefteq G; X)$ as $X \to \infty$. 
\\
\\
Counting number fields, or studying the asymptotic growth of $\mathcal{F}_{n, k}(G; X)$, is a central problem in the field of arithmetic statistics. However, for many applications, we want to count number fields with specific properties. For example, there has been substantial research into counting number fields with given local conditions \cite{DH71, DW88, Bhargava2005, Bhargava2010, Wood2010} and its applications  \cite{FreiLoughranNewton2018, McGownThorneTucker23}. Similarly, there has been substantial research on counting monogenic number fields e.g. see \cite[Introduction]{Koenig18}. In this paper, we focus on the problem of counting number fields with fixed subfields, building on the work in \cite{Alberts2021, AlbertsOdorney2021, AlbertsLemkeOliverWangWood2025}
\\
\\
Alberts formulates Conjecture 1.1 with explicit predictions for the values of $a$ and $b$ in \cite[Conjecture 3.10]{Alberts2021}. Concretely, let $\ind: S_n \to \mathbb{Z}$ be the index function given by $\ind(g) = n - \#\{\text{orbits of }g\}$, and let $a(G) = \min_{g \in G - \{1\}} \ind(g)$. Now, suppose that $\psi\in \Sur(G_k, G)$. Then, define $T(\psi)$ to be the group $T$ together with the Galois action $x.t = \psi(x) t \psi(x)^{-1}$. Observe that when $T$ is abelian, the module $T(\psi)$ only depends on the pushforward $q_*\psi$. So, later on in the paper, we sometimes abuse notation and write $T(q_*\psi)$ instead of $T(\psi)$. Given $\psi \in \Sur(G_k, G)$, define $b(k, T(\psi))$ to be the number of orbits of minimal index elements $\{t \in T: \ind(t) =a(T)\}$ under the Galois action $x: t \to (\pi(x)t\pi^{-1}(x))^{\chi^{-1}(x)}$ for $\chi: G_k \to \hat{\mathbb{Z}}^{\times}$ the cyclotomic character. Finally, define \[ B(k, T(\psi)) = \max_{\substack{N \trianglelefteq G \\ N \trianglelefteq T \\ a(N) = a(T)}} \max_{\substack{\varphi: G_k \to G\\ \varphi \equiv \psi \mod N}} b(k, N(\varphi)) .\] Then, Alberts' formulation of the Twisted Malle's Conjecture is the following: 

\begin{conjecture}[Twisted Malle's Conjecture with explicit predictions for $a$ and $b$]
Let $k$ be a number field, and $G$ a transitive permutation group of degree $n$, and $T \trianglelefteq G$ a proper normal subgroup with canonical quotient map $q : G \to G/T$, and $\pi \in q_* \Sur(G_k, G)$. Then there exists a positive constant $c > 0$ depending on $k, G, T$, and $\pi$ such that
\[
\#\{\psi \in q_*^{-1}(\pi) : |{\disc_G(\psi)}| \le X\} \sim c X^{1/a(T)} (\log X)^{B(k, T(\pi))-1}
\]
as $X \to \infty$.

\end{conjecture}

Alberts and O'Dorney proved Conjecture 1.1 for $T$ abelian in \cite[Corollary 1.2]{AlbertsOdorney2021} as long as there exists at least one $G$-extension of $k$. Moreover, the values of $a$ and $b$ found by Alberts and O'Dorney in \cite[Corollary 1.2]{AlbertsOdorney2021} agree with the prediction in Conjecture 1.2. This was the only case in which Conjecture 1.1 was known up until 2025. Then, in \cite[Theorem 3.1]{AlbertsLemkeOliverWangWood2025} Alberts, Lemke Oliver, Wang and Wood proved Conjecture 1.1 for $T = S_3^m \trianglelefteq  S_3 \wr B = G$, where $B$ is a transitive permutation group of degree $m$. In this paper, we prove many new cases of Conjecture 1.1. It is not feasible to provide an exhaustive list of all the groups for which our main results imply Conjecture 1.1. Therefore, before we state our main results, we present a few families of groups for which our results imply Conjecture 1.1. 

\begin{corollary}
Let $k$ be a number field. Fix a prime $p \in \mathbb{Z}$. Let $A$ be a finite abelian $p$-group, and $B$ a transitive permutation group of degree $m$ that is also a finite abelian $p$-group. Let $M$ be a proper subgroup of $B$, viewed as a (possibly intransitive) permutation group of degree $m$. Then, Conjecture 1.1 holds for $T = A \wr M \trianglelefteq A \wr B = G$. 
\end{corollary}

\textbf{Note:} If $A \wr B = A^m \rtimes_{\varphi} B$, then by $A \wr M \trianglelefteq A \wr B$ we specifically mean $A^m \rtimes_{\varphi|_M } M \trianglelefteq A \wr B$. This note also applies to Corollaries 1.4 and 1.5. \\ \\ The simplest new case of Conjecture 1.1 arising from the Corollary 1.3 is $T = C_2 \wr C_2 \trianglelefteq C_2 \wr C_4 = G$. As indicated in the note, $T = C_2 \wr C_2 = C_2^4 \rtimes C_2$, where the latter $C_2$ is an intransitive permutation group acting on four letters as the subgroup $C_2 < C_4 < S_4$. The above corollary in fact implies Conjecture 1.1 for all $T = C_{p^r} \wr C_{p^s} \trianglelefteq C_{p^r} \wr C_{p^{t}} = G$ with $s < t$ over any number field $k$. 

\begin{corollary}
Let $k$ be a number field. Fix a prime $p \in \mathbb{Z}$. Let $A$ be a finite abelian $p$-group, and $B$ a transitive permutation group of degree $m$ that is also a non-abelian $p$-group. Let $M$ be a subgroup of the center of $B$, viewed as a (possibly intransitive) permutation group of degree $m$. Then, Conjecture 1.1 holds for $T = A \wr M \trianglelefteq A \wr B = G$. 
\end{corollary}

The simplest new case of Conjecture 1.1 arising from the above corollary is $T = C_2 \wr M \trianglelefteq C_2 \wr D_4 = G$, where $M \cong C_2$ is the center of $D_4$. Here, we consider $D_4$ to be a permutation group of degree 4 acting on the four vertices of a square. Thus, $T = C_2 \wr M = C_2^4 \rtimes M$, where $M$ is an intransitive permutation group acting on four letters as the subgroup $M < D_4 < S_4$. 

\begin{corollary}
Let $k$ be a number field. Fix $r \in \mathbb{Z}$, and let $B$ be a transitive permutation group of degree $m$ such that $|B|$ is odd. Let $M$ be a proper abelian normal subgroup of $B$, viewed as a (possibly intransitive) permutation group of degree $m$. Then, Conjecture 1.1 holds for $T = C_r \wr M \trianglelefteq C_r \wr B = G$. 
\end{corollary}

The simplest new case of Conjecture 1.1 arising from the above corollary is $T = C_2 \wr C_3 \trianglelefteq C_2 \wr C_9 = G$. Here we consider $C_9$ to be a permutation group of degree 9 that cyclically permutes 9 elements. Thus, we have that $T = C_2 \wr C_3 = C_2^9 \rtimes C_3$, where $C_3$ is an intransitive permutation group acting on 9 letters as the subgroup $C_3 < C_9 < S_9$. 
\\
\\ 
The above families of groups are very far from being an exhaustive list of the new cases of Conjecture 1.1 implied by our main results. We will now state the main theorems proved in our paper. The definition of tower type is given at the beginning of Section 3. 

\begin{theorem}
Let $k$ be a number field. Let $G$ be an imprimitive permutation group of tower type $(A, B)$, where $A$ is a finite abelian group and $B$ is a transitive permutation group of degree $m$. Suppose that $G$ is nilpotent. Let $N = A^m \cap G$. Let $T \trianglelefteq G$ be a proper normal subgroup such that $N \subseteq T$ and $T/N$ is abelian. Let $q: G \to G/T, p: G \to G/N, h: G/N \to G/T$ denote the canonical quotient maps. If \[ \frac{1}{a(T/N)} < \frac{|A|}{a(N)} ,\] then for every $\pi \in q_* \Sur(G_k, G)$, there exists a positive constant $c$ depending on $k, G, T, $ and $\pi$ such that \[
\#\{\psi \in q_*^{-1}(\pi) : |{\disc_G(\psi)}| \le X\} \sim c X^{1/a(N)} (\log X)^{\max_{\pi_1}b(k, N(\pi_1)) -1}
\] where the maximum is taken over the set $\{\pi_1 \in p_*\Sur(G_k, G): h \circ \pi_1 = \pi\}$. 
\end{theorem}

\begin{theorem}
Let $k$ be a number field. Let $G$ be an imprimitive permutation group of tower type $(A, B)$, where is a finite abelian group and $B$ is a transitive permutation group of degree $m$. Let $r$ denote the rank of $A$. Suppose that there is at least one $G$-extension of $k$. Let $N = A^m \cap G$. Let $T \trianglelefteq G$ be a proper normal subgroup such that $N \subseteq T$ and $T/N$ is abelian. Let $q: G \to G/T, p: G \to G/N, h: G/N \to G/T$ denote the canonical quotient maps. If \[ \frac{1}{a(T/N)} + \frac{r}{2} < \frac{|A|}{a(N)} ,\] then for every $\pi \in q_* \Sur(G_k, G)$, there exists a positive constant $c$ depending on $k, G, T, $ and $\pi$ such that \[
\#\{\psi \in q_*^{-1}(\pi) : |{\disc_G(\psi)}| \le X\} \sim c X^{1/a(N)} (\log X)^{\max_{\pi_1}b(k, N(\pi_1)) -1}
\] where the maximum is taken over the set $\{\pi_1 \in p_*\Sur(G_k, G): h \circ \pi_1 = \pi\}$.
\end{theorem}

Now, we will briefly explain the approach used to prove the above theorems. Given $\pi \in q_* \Sur(G_k, G)$, we first consider the set of homomorphisms $\pi_1 \in p_* \Sur(G_k, G)$ such that $h \circ \pi_1 = \pi$. Then, for each such $\pi_1$, we determine the asymptotic growth of $\#\{ \psi \in p_*^{-1}(\pi_1): |{\disc_G(\psi)}| \leq X\}$ using the result in \cite[Corollary 1.2]{AlbertsOdorney2021} and the fact that $N$ is abelian. After this, we sum over all $\pi_1 \in p_* \Sur(G_k, G)$ such that $h \circ \pi_1 = \pi$. Since we take $T/N$ to be abelian, we then use Twisted Malle's conjecture for $T/N \trianglelefteq G/N$ to determine when the sum of the constant terms of the asymptotics converge. It is more insightful to describe the approach in terms of counting number fields. Let $L/k$ be the number field corresponding to $\pi$ under the Galois correspondence. Then, $L/k$ is a $G/T$-extension. Recall that we want to count the number of $G$-extensions $F/k$ with $|{\disc(F/k)}| \leq X$ such that the subfield of $\tilde{F}$ fixed by $T$ is $L$. In order to do this, we first consider the set $U$ of $G/N$ extensions $M/k$ such that the subfield of $\tilde{M}$ fixed by $T/N$ is $L$. We can count the number of such $G/N$ extensions by using Twisted Malle's Conjecture for $T/N \trianglelefteq G/N$. Then, for each such $G/N$-extension $M/k$, we count the number of $G$-extensions $F/k$ with $|{\disc(F/k)}| \leq X$ such that the subfield of $\tilde{F}$ fixed by $N$ is $M$ using Twisted Malle's conjecture for $N \trianglelefteq G$. Summing over the $G/N$-extensions in $U$ gives us asymptotics for the number of $G$-extensions $F/k$ with $|{\disc(F/k)}| \leq X$ such that the subfield of $\tilde{F}$ fixed by $T$ is $L$, as desired. 
\\
\\
We can check that the value of $a$ given by Theorems 1.6 and 1.7 agrees with the value of $a$ predicted by Alberts in Conjecture 1.2. Since $N \subseteq T$, we have $a(T) \leq a(N)$. Now, fix an element $g \in T$. Since $T \trianglelefteq G \subseteq A \wr B$, we may write $g = ((g_1, \dots, g_m), g_{m + 1})$, where $g_1, \dots, g_m \in A$ and $g_{m + 1} \in T/N \trianglelefteq B$. If $g_{m + 1} = e_B$, then we have $g \in N$, which implies that $\ind(g) \geq a(N)$. If $g_{m + 1} \neq e_B$, then the number of orbits (or cycles) of $g_{m + 1} \in B \subset S_m$ is $\leq m - a(T/N)$. Thus, the number of orbits (or cycles) of $g \in T \subset S_{m|A|}$ is $\leq m|A| - a(T/N)|A|$. Therefore, $\ind(g) \geq a(T/N)|A| > a(N)$, where the last inequality follows from the hypotheses of Theorems 1.6 and 1.7. Hence, we have shown that if $g \in T$, then $\ind(g) \geq a(N)$. Thus, it follows that $a(T) \geq a(N)$. Hence, we have $a(T) = a(N)$, as desired. Thus, we see that the value of $a$ given by Theorems 1.6 and 1.7 ($a(N)$) agrees with Alberts' prediction in Conjecture 1.2 ($a(T)$). \\ \\ Note that the value of $b$ given by Theorems 1.6 and 1.7 is clearly less than or equal to $B(k, T(\pi))$ since $B(k, T(\pi))$ is a maximum over all $N \trianglelefteq G$ with $N \trianglelefteq T$ and $a(N) = a(T)$, not just $N = A^m \cap G$. Now, to prove the other direction, suppose that $N_1 \trianglelefteq G$ is a proper normal subgroup such that $N_1 \trianglelefteq T$ and $a(N_1) = a(T)$. We conjecture that the value of $b$ given by Theorems 1.6 and 1.7 is equal to $B(k, T(\pi))$, which is a purely group-theoretic conjecture. 
\\
\\
In section 2 we provide an inductive framework, modeled on \cite{AlbertsLemkeOliverWangWood2025}, which allows us to deduce Conjecture 1.1 for $T \trianglelefteq G$ provided one can obtain suitable explicit bounds on the sizes of the unramified cohomology groups $H^1_{ur}(k, N(\pi_1))$ for an abelian normal subgroup $N \trianglelefteq G$ with $N \subseteq T$. In section 3, we implement this framework to prove Conjecture 1.1 in a wide range of new cases. 

\section*{Acknowledgements}

The author thanks Professor Melanie Wood for advising this project and giving many helpful suggestions. The author would also like to thank Brandon Alberts and Evan O'Dorney for providing helpful comments on an earlier draft of the paper.  

\section{Preliminary Work}

Let $k$ be a number field, and let $G$ be a finite transitive permutation group of degree $n$. Let $T \trianglelefteq G$ be a proper normal subgroup with canonical quotient map $q: G \to G/T$. Let $N \trianglelefteq G$ be a proper normal subgroup with canonical quotient map $p: G \to G/N$. Suppose that $N$ is abelian and that $N \subseteq T$. Moreover, let $h: G/N \to G/T$ denote the canonical quotient map. Given $\pi \in q_*\Sur(G_k, G)$, define \[ S(\pi) = \{\pi_1 \in p_*\Sur(G_k, G): h \circ \pi_1 = \pi \}.\] Now, fix $\pi \in q_*\Sur(G_k, G)$. Then, we have that \[ \#\{\psi \in q_*^{-1}(\pi) : |{\disc_G(\psi)}| \le X\} = \sum_{\pi_1 \in S(\pi)} \#\{\psi \in p_*^{-1}(\pi_1): |{\disc_G(\psi)}| \le X\} .\] In the above equation we are essentially counting the $G$-extensions with a given $T$ subfield $L$ by summing over all $G/N$-extensions whose $T/N$ subfield is $L$.     
\\
\\
The proposition below is essentially a modified version of Theorem 1.11 in \cite{AlbertsLemkeOliverWangWood2025}, tailored to be more effective for proving cases of Conjecture 1.1. We define the cohomology group $H^1_{ur}(k,T(\pi))$ to be the subgroup of unramified classes,
\[
H^1_{ur}(k,T(\pi)) \;=\; \{ f \in H^1(k,T(\pi)) : \forall \mathfrak{p},\; \mathrm{res}_{I_{\mathfrak{p}}}(f) = 0 \},
\]
where $\mathfrak{p}$ ranges over all finite and infinite places of $k$, and $I_{\mathfrak{p}}$ is the inertia group of $k$ at each finite place $\mathfrak{p}$ and the decomposition group at each infinite place.
\begin{proposition}
Suppose that there exists at least one $G$-extension of $k$. Now, Suppose $\theta \geq 0$ is such that $\theta < \frac{1}{a(N)}$ and \[ \sum_{\pi_1 \in S(\pi) \cap p_*\Sur(G_k, G; X)} |H^1_{ur}(k, N(\pi_1))| \ll_{k, n, \pi} X^{\theta}.\] Then, there exists a positive constant $c$ depending on $k, G, T,$ and $\pi$ such that \[
\#\{\psi \in q_*^{-1}(\pi) : |{\disc_G(\psi)}| \le X\} \sim c X^{1/a(N)} (\log X)^{\max_{\pi_1}b(k, N(\pi_1)) -1}
\] where the maximum is taken over $\pi_1 \in S(\pi)$.
\end{proposition}

\begin{proof}

In \cite[Corollary 1.2]{AlbertsOdorney2021}, Alberts and O'Dorney proved Conjecture 1.1 in the case where $T$ is abelian given that there exists at least one $G$-extension of $k$. Thus, since $N$ is abelian, we have that for each $\pi_1 \in S(\pi)$, there exists a positive constant $c'(\pi_1)$ such that \[\#\{\psi \in p_*^{-1}(\pi_1): |{\disc_G(\psi)}| \le X\} \sim c'(\pi_1) X^{1/a(N)} (\log X)^{b(k, N(\pi_1)) - 1} .\] Now, define $b = \max_{\pi_1} b(k, N(\pi_1))$. If $b(k, N(\pi_1)) = b$, then define $c(\pi_1) = c'(\pi_1)$. Otherwise, define $c(\pi_1) = 0$. Then, we have that for each $\pi_1 \in S(\pi)$, there exists a constant $c(\pi_1) \geq 0$ such that \[ \#\{\psi \in p_*^{-1}(\pi_1): |{\disc_G(\psi)}| \le X\} \sim c(\pi_1) X^{1/a(N)} (\log X)^{b - 1} \] Dividing both sides of the equation by $X^{1/a(N)} (\log X)^{b - 1}$, we get that \begin{equation} \lim_{X \to \infty} \frac{\#\{\psi \in p_*^{-1}(\pi_1): |{\disc_G(\psi)}| \le X\}}{X^{1/a(N)} (\log X)^{b - 1}} = c(\pi_1)\end{equation} for each $\pi_1 \in S(\pi)$.  Now, it follows from Theorem 6.1 in \cite{AlbertsLemkeOliverWangWood2025} that \[ \#\{\psi \in p_*^{-1}(\pi_1): |{\disc_G(\psi)}| \le X\} =
O_{n,[k:\mathbb{Q}],\epsilon}\!\left(
  \frac{\lvert H^1_{ur}(k, N(\pi_1)) \rvert}
       {(p_* \disc_G(\pi_1))^{1/a(N) - \epsilon}}
  \, X^{1/a(N)} (\log X)^{b(k,N(\pi_1))-1}
\right)
\] for all $\pi_1 \in S(\pi)$. In particular, this implies that for each $\pi_1 \in S(\pi)$, there exists a constant $f(\pi_1) $ with \[ f(\pi_1) \ll_{n, [k: \mathbb{Q}], \epsilon} \frac{|H^1_{ur}(k, N(\pi_1))|}{(p_* \disc_G(\pi_1))^{1/a(N) - \epsilon}}   \] such that \[ \#\{\psi \in p_*^{-1}(\pi_1): |{\disc_G(\psi)}| \le X\} \leq f(\pi_1) X^{1/a(N)} (\log X)^{b - 1} \] for every $X \geq 2$. Again, dividing by $X^{1/a(N)} (\log X)^{b - 1}$ we get that \begin{equation} \frac{\#\{\psi \in p_*^{-1}(\pi_1): |{\disc_G(\psi)}| \le X\}}{X^{1/a(N)} (\log X)^{b - 1}} \leq f(\pi_1) \end{equation} for every $X \geq 2$. Now, recall that there exists some $\theta \geq 0$ such that $\theta < \frac{1}{a(N)}$ and \[ \sum_{\pi_1 \in S(\pi) \cap p_*\Sur(G_k, G; X)} |H^1_{ur}(k, N(\pi_1))| \ll X^{\theta}.\] From this, we get that \[ \sum_{\pi_1 \in S(\pi) \cap p_*\Sur(G_k, G; X)} f(\pi_1) \ll_{n, k, \epsilon} \sum_{\pi_1 \in S(\pi) \cap p_*\Sur(G_k, G; X)} \frac{|H^1_{ur}(k, N(\pi_1))|}{(p_*\disc_G(\pi_1))^{1/a(N) - \epsilon}} \ll_{n, k, \epsilon} 1 + X^{\theta - 1/a(N) + \epsilon}  \] (We get the second inequality by splitting into dyadic ranges). Since $\theta < \frac{1}{a(N)}$, the above implies that the series \[ \sum_{\pi_1 \in S(\pi)} f(\pi_1) \] converges. Observe that $c(\pi_1) \leq f(\pi_1)$ for all $\pi_1 \in S(\pi)$. This implies that the series $\sum_{\pi_1 \in S(\pi)} c(\pi)$ converges to some finite value $c$. Now, given (1), (2) and the fact that $\sum_{\pi_1 \in S(\pi)} f(\pi_1)$ converges, it follows from Tannery's theorem that \[ \lim_{X \to \infty} \sum_{\pi_1 \in S(\pi)} \frac{\#\{\psi \in p_*^{-1}(\pi_1): |{\disc_G(\psi)}| \le X\}}{X^{1/a(N)} (\log X)^{b - 1}} = \sum_{\pi_1 \in S(\pi)} \lim_{X \to \infty} \frac{\#\{\psi \in p_*^{-1}(\pi_1): |{\disc_G(\psi)}| \le X\}}{X^{1/a(N)} (\log X)^{b - 1}} \]\[ =  \sum_{\pi_1 \in S(\pi)} c(\pi_1) = c  .\] Thus, we get \[ \lim_{X \to \infty} \frac{ \#\{\psi \in q_*^{-1}(\pi) : |{\disc_G(\psi)}| \le X\}}{X^{1/a(N)} (\log X)^{b - 1}}  = c \] as desired.

\end{proof}

\section{Main results}

Let $A$ be a finite abelian group and let $B$ be a transitive permutation group of degree $m$ such that there is at least one $B$-extension of $k$. We call a group $H$ an imprimitive permutation group with tower type $(A, B)$ if we have $H \subset A \wr B$ with $H$ surjecting onto $B$ and $A^m \cap H$ surjecting onto $A$ through each of the $m$ projection maps. Now, suppose that $G$ is an imprimitive permutation group with tower type $(A, B)$. Let $N = A^m \cap G$. Then, $N$ is a proper normal subgroup of $G$. Moreover, $N$ is abelian since $A$ is abelian.
\\ 
\\
Let $T \trianglelefteq G$ be a proper normal subgroup such that $A^m \cap G = N \subseteq T$. Let $p, q, h$ be as in Section 2. Now, suppose that $T/N$ is an abelian subgroup of $G/N \cong B$. Now, fix $\pi \in h_*\Sur(G_k, B)$. Then, it follows from \cite[Corollary 1.2]{AlbertsOdorney2021} that there exists a positive constant $c$ depending on $k, B, T/N$ and $\pi$ such that \begin{equation} \#\{\pi_1 \in h_*^{-1}(\pi): |{\disc_{B}(\pi_1)}| \leq X \} \sim cX^{1/a(T/N)} (\log X)^{b(k, (T/N)(\pi)) - 1} \end{equation} Let $S = \Stab_G(1)$. Moreover, let $S'$ be the preimage of $\Stab_B(1)$ in $G$. Since $N = A^m \cap G$, it follows that $N = \cap_{g \in G} gS'g^{-1}$. Therefore, applying Proposition 5.2 from \cite{AlbertsLemkeOliverWangWood2025}, we get \[ p_* \Sur(G_k, G; X) \subseteq \Sur(G_k, B; dX^{1/|A|}) \] for some constant $d > 0$ depending only on $[k: \mathbb{Q}]$ and $n$. Hence, it follows that \[ |S(\pi) \cap p_*\Sur(G_k, G; X)| \leq \#\{\pi_1 \in h_*^{-1}(\pi): |{\disc_B(\pi_1)}| \leq dX^{1/|A|}\}.\] From this, it immediately follows that \begin{equation} |S(\pi) \cap p_*\Sur(G_k, G; X)|  \ll_{n, k, \epsilon, B, T/N, \pi} X^{(1/(|A|)a(T/N)) + \epsilon} \end{equation} 

\subsection{Special case of Nilpotent groups}

\begin{proof}[Proof of Theorem 1.6] Now, suppose that $G$ is a nilpotent group. Since $G$ is a finite nilpotent group, it follows that it is a finite solvable group. Therefore, it follows from Shafarevich's theorem that there exists a $G$-extension of $k$ \cite{SW98}. \\ \\ Since $N \trianglelefteq G$ is abelian, it follows that $N(\pi_1)$ is a nilpotent $G_k$-module for every $\pi_1 \in p_*\Sur(G_k, G)$. Therefore, it follows from Corollary 1.14(i) in \cite{AlbertsLemkeOliverWangWood2025} that \[ |H^1_{ur}(k, N(\pi_1))| \ll_{k, N, \pi_1, \epsilon} X^{\epsilon} \] for all $\pi_1 \in p_*\Sur(G_k, G; X)$. Combining this with (6), we immediately get \[ \sum_{\pi_1 \in S(\pi) \cap p_*\Sur(G_k, G; X)} |H^1_{ur}(k, N(\pi_1))| \ll_{k, n, \pi, \epsilon} X^{(1/(|A|)a(T/N)) + \epsilon} .\] In particular, if \[ \frac{1}{|A|a(T/N)} < \frac{1}{a(N)} \] then we can use Proposition 2.1 to deduce that Conjecture 1.1 holds for $T \trianglelefteq G$. In other words, we immediately get Theorem 1.6, as desired.
\end{proof}
Now, we will consider some of the implications of Theorem 1.6. As above, let $G$ be an imprimitive permutation group of tower type $(A, B)$, and let $N = A^m \cap G$. Now, further suppose that $A$ and $B$ are $p$-groups for some prime $p$. Then, clearly $A \wr B$ is also a $p$-group, and so $G \subseteq A \wr B$ is also necessarily a $p$-group. Since every finite $p$-group is nilpotent, it follows that $G$ must be nilpotent. Therefore, assuming that $A$ and $B$ are $p$-groups is useful, as it allows us to arbitrarily choose $A$ and $B$ while still ensuring that $G$ is nilpotent. In particular, we get the following corollary:

\begin{corollary}
Let $k$ be a number field. Let $p \in \mathbb{Z}$ be a prime. Let $G$ be an imprimitive permutation group of tower type $(A, B)$, where $A$ is a finite abelian $p$-group and $B$ is a transitive permutation group of degree $m$ that is also a $p$-group. Let $N = A^m \cap G$. Let $T \trianglelefteq G$ be a proper normal subgroup such that $N \subseteq T$ and $T/N$ is abelian. Let $q: G \to G/T$ denote the canonical quotient map. If \[ \frac{1}{a(T/N)} < \frac{|A|}{a(N)} ,\] then for every $\pi \in q_* \Sur(G_k, G)$, there exists a positive constant $c$ depending on $k, G, T, $ and $\pi$ such that \[
\#\{\psi \in q_*^{-1}(\pi) : |{\disc_G(\psi)}| \le X\} \sim c X^{1/a(N)} (\log X)^{\max_{\pi_1}b(k, N(\pi_1)) -1}
\] where the maximum is taken over $\pi_1 \in S(\pi)$.
\end{corollary}

Now, let $f_0: G \to B$ be the projection onto $B$ and let $f_1, \dots, f_m: G \to A$ be the projection maps onto $A$. We call an element $g \in G$ a pure element if $f_0(g) = e_B$, $f_i(g) \neq e_A$ for some $1 \leq i \leq m$, and $f_j(g) = e_A$ for all $1 \leq j \leq m$ such that $j \neq i$ (Here $e_A$ and $e_B$ denote the identity elements of $A$ and $B$ respectively).

\begin{corollary}
Let $k$ be a number field. Let $p \in \mathbb{Z}$ be a prime. Let $G$ be an imprimitive permutation group of tower type $(A, B)$, where $A$ is a finite abelian $p$-group and $B$ is a transitive permutation group of degree $m$ that is also a $p$-group. Suppose that $G$ contains a pure element. Let $N = A^m \cap G$. Let $f: G \to B$ denote the surjection onto $B$. Let $T'$ be a proper abelian normal subgroup of $B$, and let $T = f^{-1}(T')$. Let $q: G \to G/T$ denote the canonical quotient map. Then, for every $\pi \in q_* \Sur(G_k, G)$, there exists a positive constant $c$ depending on $k, G, T,$ and $\pi$ such that \[
\#\{\psi \in q_*^{-1}(\pi) : |{\disc_G(\psi)}| \le X\} \sim c X^{1/a(N)} (\log X)^{\max_{\pi_1}b(k, N(\pi_1)) -1}
\] where the maximum is taken over $\pi_1 \in S(\pi)$. 
\end{corollary}

\begin{proof}
Observe that $N = \ker(f)$. Therefore, we have that $N \subseteq T$ and $T/N \cong T'$, which is abelian. Let $g \in G$ be a pure element. Then, $g$ fixes every element in $m - 1$ distinct copies of $A$. In particular, we have $\ind(g) < m|A| - (m - 1)|A| = |A|$. Therefore, $a(N) = a(A^m \cap G) < |A|$, which implies that $\frac{|A|}{a(N)} > 1$. Since $\frac{1}{a(T/N)} \leq 1$, it immediately follows that $\frac{1}{a(T/N)} < \frac{|A|}{a(N)}$. This completes the proof of the corollary. 
\end{proof}

Note that Corollary 3.2 immediately gives a wide class of examples of groups $T \trianglelefteq G$ for which Conjecture 1.1 holds. For example, observe that $A \wr B$ clearly contains a pure element. Thus, letting $G = A \wr B$ in Corollary 3.2, we immediately get the following result: 

\begin{corollary}
Let $k$ be a number field. Let $p \in \mathbb{Z}$ be a prime. Let $A$ be a finite abelian $p$-group, and let $B$ be a transitive permutation group of degree $m$ that is a $p$-group. Let $G = A \wr B = A^m \rtimes_{\varphi} B$. Let $M$ be a proper abelian normal subgroup of $B$, and let $T = A \wr M = A^m \rtimes_{\varphi|_M} M \subset G$. Let $q: G \to G/T$ denote the canonical quotient map. Then, for every $\pi \in q_*\Sur(G_k, G)$, there exists a positive constant $c$ depending on $k, G, T,$ and $\pi$ such that \[
\#\{\psi \in q_*^{-1}(\pi) : |{\disc_G(\psi)}| \le X\} \sim c X^{1/a(N)} (\log X)^{\max_{\pi_1}b(k, N(\pi_1)) -1}
\] where $N = A^m$ and where the maximum is taken over $\pi_1 \in S(\pi)$. 
\end{corollary}

Observe that Corollary 3.3 immediately implies Corollaries 1.3 and 1.4. 

\subsection{General case}

\begin{proof}[Proof of Theorem 1.7]Observe that the subgroup $A^m \leq A \wr B$ has the induced module structure by definition. In other words, we have \[ A \wr B = \Ind^B_1(A) \rtimes B .\] Now, fix $\pi_1 \in S(\pi) \cap p_*\Sur(G_k, G; X) \subset \Sur(G_k, B)$. Let $F$ be the subfield of $\overline{k}$ fixed by $\pi_1^{-1}(\Stab_B 1)$. Then, $F$ is a degree $m$ extension of $k$ with Galois closure group $B$ and $|{\disc(F/k)}| = |{\disc_B(\pi_1)}| < dX^{1/|A|} $. Moreover, we necessarily have the $G_k$-module isomorphism \[ \Ind_1^B(A)(\pi) = \Ind_F^k(A) \] where $A$ carries the trivial $G_F$ action. Hence, we have an injective homomorphism of $G_k$ modules $N(\pi_1) \hookrightarrow \Ind_F^k(A)$. Now, Corollary 1.14(iii) gives us \[ |H^1_{ur}(k, N(\pi_1))| \ll_{m, |A|} |{\Hom(\Cl_F, A)}| .\] Now, let $A$ be a finite abelian group, and let $r$ denote the rank of $A$. Then, it is well-known that \[ |{\Hom(\Cl_F, A)}| \ll_{\epsilon} |{\disc(F/\mathbb{Q})}|^{r/2 + \epsilon} = (|{\disc(k/\mathbb{Q})}|^m|{\disc(F/k)}|)^{r/2 + \epsilon} \ll_{n, k, \epsilon, A} X^{\frac{r}{2|A|}+ \epsilon} .\] Combining this with (6), we immediately get that \[ \sum_{\pi_1 \in S(\pi) \cap p_*\Sur(G_k, G; X)} |H^1_{ur}(k, N(\pi_1))| \ll_{k, n, \pi, \epsilon} X^{\frac{1}{|A|a(T/N)} + \frac{r}{2|A|} + \epsilon} .\] In particular, if \[ \frac{1}{a(T/N)} + \frac{r}{2} < \frac{|A|}{a(N)} \] then we can use Proposition 2.1 to deduce that Conjecture 1.1 holds for $T \trianglelefteq G$. In other words, we immediately get Theorem 1.7, as desired. 
\end{proof}

Now, we will consider some of the implications of Theorem 1.7. Observe that when $A \cong C_t$ for some $t \in \mathbb{Z}$, the condition in Theorem 1.7 becomes \[ \frac{1}{a(T/N)} + \frac{1}{2} < \frac{|A|}{a(N)} = \frac{t}{a(N)} .\]

\begin{corollary}
Let $k$ be a number field. Let $G$ be an imprimitive permutation group of tower type $(A, B)$, where $A = C_t$ for some $t \in \mathbb{Z}$ and $B$ is a transitive permutation group of degree $m$. Suppose that there is at least one $G$-extension of $k$. Further suppose that $G$ contains a pure element. Let $N = A^m \cap G$. Let $f: G \to B$ denote the surjection onto $B$. Let $T'$ be a proper abelian normal permutation subgroup of $B$ that does not contain any transposition in $S_m$, and let $T = f^{-1}(T')$. Let $q: G \to G/T$ denote the canonical quotient map. Then, for every $\pi \in q_* \Sur(G_k, G)$, there exists a positive constant $c$ depending on $k, G, T, $ and $\pi$ such that \[
\#\{\psi \in q_*^{-1}(\pi) : |{\disc_G(\psi)}| \le X\} \sim c X^{1/a(N)} (\log X)^{\max_{\pi_1}b(k, N(\pi_1)) -1}
\] where the maximum is taken over $\pi_1 \in S(\pi)$.
\end{corollary}

\begin{proof}
Since $A \cong C_t$, it follows that $r = \rk A = 1$. Let $g \in G$ be a pure element. Then, $g$ fixes every element in $m - 1$ distinct copies of $A$. Thus, we have $\ind(g) < m|A| - (m - 1)|A| = |A|$. Therefore, $a(N) = a(A^m \cap G) < |A|$, which implies that $\frac{|A|}{a(N)} > 1$. Let $T'$ be a proper abelian normal permutation subgroup of $B$ that does not contain any transposition in $S_m$. Then, $a(T') \geq 2$. Let $f: G \to B$ denote the surjection onto $B$, and let $T = f^{-1}(T')$. Then, it follows that $N = \ker(f) \subset T$ and $T/N \cong T'$, which is abelian. In particular, since $a(T/N) \geq 2$ and $\frac{|A|}{a(N)} > 1$, we have $\frac{1}{a(T/N)} + \frac{1}{2} < \frac{|A|}{a(N)}$. 
\end{proof}

Observe that Corollary 3.4 gives a wide class of examples of groups $T \trianglelefteq G$ for which Conjecture 1.1 holds. Observe that $A \wr B$ clearly contains a pure element. Moreover, if $|B|$ is odd, then $B$ clearly cannot contain any transposition in $S_m$. Note that it follows from the Feit-Thompson theorem that every finite group of odd order is solvable. Thus, if $|B|$ is odd, then $B$ is solvable. In this case, $A = C_t$ is abelian and $B$ is solvable, so $G$ is solvable. Then, it follows from Shafarevich's theorem that there exists a $G$-extension of $k$ \cite{SW98}. Therefore, we get the following class of examples (which is very far from being exhaustive): 

\begin{corollary}
Let $k$ be a number field. Let $r \in \mathbb{Z}$, and let $B$ be a transitive permutation group of degree $m$ such that $|B|$ is odd. Let $G 
= C_r \wr B = C_r^m \rtimes_{\varphi} B$. Let $M$ be a proper abelian normal subgroup of $B$, and let $T = C_r \wr M = C_r^m \rtimes_{\varphi|_M} M \subset G$. Let $q: G \to G/T$ denote the canonical quotient map. Then, for every $\pi \in q_* \Sur(G_k, G)$, there exists a positive constant $c$ depending on $k, G, T, $ and $\pi$ such that \[
\#\{\psi \in q_*^{-1}(\pi) : |{\disc_G(\psi)}| \le X\} \sim c X^{1/a(N)} (\log X)^{\max_{\pi_1}b(k, N(\pi_1)) -1}
\] where $N = C_p^m$ and where the maximum is taken over $\pi_1 \in S(\pi)$. 
\end{corollary}

Observe that Corollary 3.5 immediately gives us Corollary 1.5, as desired.

\end{document}